\documentclass[12pt,a4paper]{amsart}

\usepackage[utf8]{inputenc}
\usepackage{enumerate,amssymb,ushort,bbm,titletoc,a4,mathtools}
\usepackage[usenames,dvipsnames]{xcolor}
\usepackage[pagebackref,colorlinks,linkcolor=BrickRed,citecolor=OliveGreen,urlcolor=black,hypertexnames=true]{hyperref}

\DeclareMathOperator{\rk}{rk}
\DeclareMathOperator{\disc}{disc}
\DeclareMathOperator{\res}{res}

\DeclareMathOperator{\UD}{UD}
\DeclareMathOperator{\dom}{dom}

\DeclareMathOperator{\opm}{M}

\newcommand{\de}{\delta}

\newcommand{\N}{\mathbb{N}}

\newcommand{\cO}{\mathcal{O}}

\newcommand{\ti}{{\rm t}}
\newcommand{\rr}{r}
\newcommand{\kk}{\mathbb{K}}

\newcommand{\gx}{\xi}
\newcommand{\gX}{\Xi}

\newcommand*{\mat}[1]{\opm_{#1}(\kk)}
\newcommand*{\GL}[1]{\operatorname{GL}_{#1}(\kk)}
\newcommand*{\SL}[1]{\operatorname{SL}_{#1}(\kk)}
\newcommand*{\slie}[1]{\mathfrak{sl}_{#1}(\kk)}
\newcommand*{\ud}[1]{\UD_{#1}}

\newcommand{\Langle}{\mathop{<}\!}
\newcommand{\Rangle}{\!\mathop{>}}
\newcommand{\ulx}{\underline{x}}
\newcommand{\ulg}{\underline{\gx}}

\newcommand{\px}{\kk\!\Langle \ulx\Rangle}
\makeatletter
\def\moverlay{\mathpalette\mov@rlay}
\def\mov@rlay#1#2{\leavevmode\vtop{
		\baselineskip\z@skip \lineskiplimit-\maxdimen
		\ialign{\hfil$#1##$\hfil\cr#2\crcr}}}
\makeatother

\newcommand{\plangle}{\moverlay{(\cr<}}
\newcommand{\prangle}{\moverlay{)\cr>}}
\newcommand{\rx}{\kk\plangle \ulx \prangle}

\newtheorem{thm}{Theorem}[section]
\newtheorem{lem}[thm]{Lemma}
\newtheorem{cor}[thm]{Corollary}
\newtheorem{prop}[thm]{Proposition}
\theoremstyle{definition}

\newtheorem{exa}[thm]{Example}
\theoremstyle{remark}
\newtheorem{rem}[thm]{Remark}

\newtheorem{question}[thm]{Question}

\begin{document}
\title[Matrix evaluations of nc rational  functions and Waring problems]{Matrix evaluations of noncommutative rational functions and Waring problems}

\author{Matej Bre\v{s}ar}
\address{Faculty of Mathematics and Physics, University of Ljubljana \&
Faculty of Natural Sciences and Mathematics, University of Maribor \& 
Institute of Mathematics, Physics, and Mechanics, Slovenia}
\email{matej.bresar@fmf.uni-lj.si}

\author{Jurij Volčič}
\address{Department of Mathematics, University of Auckland, New Zealand}
\email{jurij.volcic@auckland.ac.nz}
\thanks{M.\ Brešar was 
 supported by Grant P1-0288, ARIS (Slovenian Research and Innovation Agency).
 J.\  Volčič was supported by the NSF grant DMS-1954709.}



\begin{abstract}
Let $\rr$ be a nonconstant noncommutative rational function in $m$ variables over an algebraically closed field $\kk$ of characteristic 0. We show that for 
 $n$ large enough, there exists an $X\in \mat{n}^m$ such that $\rr(X)$ has $n$ distinct and nonzero eigenvalues. 
 This result is used to study the linear and multiplicative Waring problems for matrix algebras. Concerning the linear problem, we  show that for $n$ large enough, every matrix in 
$\slie{n}$ can be written as $\rr(Y)-\rr(Z)$
for some $Y,Z\in \mat{n}^m$. We also discuss variations of this result for the case where $\rr$ is a noncommutative polynomial. Concerning the multiplicative problem,  we show, among other results, that if $f$ and $g$  
are nonconstant polynomials, then, for $n$ large enough, every nonscalar matrix  in $\GL{n}$ 
can be written as 
$f(Y)\cdot g(Z)$ 
for some $Y,Z\in \mat{n}^m$.
\end{abstract}

\maketitle

\section{Introduction}

Throughout, $\kk$ stands for an algebraically closed field of characteristic 0. Let $\ulx=(x_1,\dots,x_m)$ be a tuple of $m$ freely noncommuting variables, let $\px$ be the free $\kk$-algebra of noncommutative polynomials in $\ulx$, and let $\rx$ be its universal skew field of fractions \cite{Coh}, the free field of noncommutative rational functions in $\ulx$. 
For example, $x_1^2+x_1x_2-x_2x_1-x_1+2$ is a noncommutative polynomial, and $(x_1x_2^{-1}x_1-x_2)^{-1}-x_2^{-1}$ is a noncommutative rational function.
Our most general results concern elements $\rr$ from  $\rx$, but   are also new, and of particular interest,  in the special case where $\rr\in\px$. 

Noncommutative rational functions originated in noncommutative algebra \cite{Ami,Coh,GGRW} and automata theory \cite{BR}. Nowadays, they prominently appear in free analysis \cite{KVV,KVV1}, free probability \cite{HMS,CMMPY} and free real algebraic geometry \cite{HMV,AHKM,Vol1}, which study analytic, geometric and asymptotic properties of their matrix  evaluations. More precisely, noncommutative rational functions are (after noncommutative polynomials) the most tractable class of so-called {\it free functions}, which are maps on tuples of matrices of all sizes that preserve direct sums and similarities.
When dealing with  $\rr\in \rx$, 
we will be interested in  the matrix evaluations
$\rr(X)$ with $X\in \mat{n}^m$ (more precisely, $X$ must belong to the domain of $\rr$) for arbitrary $n$. Here,  $\mat{n}$ stands for the algebra of $n \times n$ matrices over $\kk$. 

Our fundamental result, from which all else follows, states that given a  nonconstant  $\rr\in\rx$, 
for every  large enough $n$ there exists an $X$
in $\mat{n}^m$ such that $\rr(X)$ has $n$ distinct eigenvalues (Theorem \ref{t:dist}). Moreover, these eigenvalues can be chosen to be nonzero  
 (Corollary \ref{t:dist1}). 
The proof of this result consists of two natural parts. Firstly, one investigates the image of a noncommutative rational function on matrices of a fixed prime size using tools from central simple algebras and algebraic geometry. Then one uses ampliation technique inherent in free analysis along with classical number theory to glue together the preceding size-dependent conclusions.

Our result on eventually attaining images with pairwise distinct eigenvalues conforms organically with the perspective of free analysis and free probability, which investigate those features of matrix evaluations of noncommutative rational functions that persist for matrices of large sizes.
On the other hand, 
our main motivation behind this result is its applicability to  Waring  problems for matrix algebras.  We explain this in the next paragraphs.

The classical Waring’s problem,  proposed by  Waring in the 18\textsuperscript{th} century and solved affirmatively by Hilbert in 1909,
asks 
whether for every  natural number $k$ there exists a natural number $g(k)$ such that every natural number is a sum of  $g(k)$ $k$th powers.
Various variations of this problem occur in different areas of mathematics, including  
 group theory. 
In \cite{LST}, Larsen, Shalev, and Tiep 
 proved  that for any word  $w \ne 1$,  $w(G)^2 =G$ for every
 finite non-abelian simple group $G$ of sufficiently high order (here, 
$w(G)$ denotes 
  the image of the word map 
 induced by $w$).
More generally, if $w,w'$ are two nontrivial words, then $w(G)w'(G)=G$.
 
 Seeking analogs of this result in noncommutative algebra, it is natural to replace the role of words by noncommutative polynomials, and the role of finite simple groups by matrix algebras. This type of Waring problem was initiated in \cite{B} and discussed further  in 
  \cite{ BS, BSb, Chen, PP}.
    More specifically, the following question, raised as Question 4.8 in \cite{BS} and studied (but not solved) in \cite{BSb},  emerged as central: 
Given a nonconstant polynomial $f\in \px$, does  
$f(\mat{n}) - f(\mat{n})$ contain 
the Lie algebra of trace zero matrices $\slie{n}$ for all sufficiently large $n$? Here,  $f(\mat{n})$ stands for the image of $f$ in $\mat{n}$ (i.e., $f(\mat{n})$ is the set of all $ f(A_1,\dots,A_m)$ with $ A_i\in \mat{n}$), and 
$f(\mat{n}) - f(\mat{n})$ for 
$\{B-B': B,B'\in f(\mat{n})\}$. As a byproduct of the  result on distinct eigenvalues discussed above, we will solve this problem, not only for noncommutative polynomials but for noncommutative rational functions; see Theorem \ref{c:diff}. In later parts of Section  \ref{s3}, we  discuss extensions and variations of this result. In particular, we consider the question of when matrices with nonzero trace are linear combinations of two matrices from  $f(\mat{n})$ (see Theorem \ref{c23}). Also, we obtain (partial) extensions of our results from matrices over $\kk$ to algebras $\mathcal A$  satisfying $A\cong \opm_2(\mathcal A)$ (see Theorem \ref{cbx}). 

The problem of presenting matrices as linear combinations of matrices from the image of a rational function or  polynomial 
 may be called the {\it linear Waring problem}.
The fundamental theorem on eigenvalues also enables  us to consider a different,  
 {\it multiplicative Waring problem} for polynomials and rational functions, which is even closer to the group-theoretic one and has, to the best of our knowledge,  not yet been studied in the literature. This problem is the topic of Section \ref{s4}. One of our results, Theorem \ref{c:prod},  states that if $f,g\in \px$ are nonconstant polynomials, then, for $n$ large enough, every nonscalar matrix  in $\GL{n}$ is contained in 
$ f(\mat{n})\cdot g(\mat{n}).$ 
We also mention Theorem \ref{cza} which in particular shows that 
if $f\in \px$ is a nonzero polynomial with zero constant term, 
 then   for $n$ large enough, every matrix in
$\mat{n}$ is a product of twelve matrices from $f(\mat{n})$.

We do not know at present whether  scalar matrices  should really be excluded in Theorem \ref{c:prod}. This is one of the problems that are left open. It is interesting that a genuinely different problem that arises in Section \ref{s3} also involves scalar matrices (see Question \ref{q36}). We also do not know whether Theorem \ref{cza}  holds for nonconstant polynomials
with nonzero constant term  (and, on the other hand, we believe that the number twelve can be lowered).

\subsection*{Acknowledgments}
The authors thank Igor Klep and Peter Šemrl for insightful discussions.

\section{Eigenvalues of matrix evaluations of noncommutative rational functions }

For a comprehensive study of the skew field of noncommutative rational functions, also called the free skew field, we refer to \cite{Coh}. 
Analogous to how a (usual) rational function can be given as a fraction of different pairs of polynomials, a noncommutative rational function can be given by different formal rational expressions. For example, the expressions $x_2(1-x_1x_2)^{-1}x_2$ and $(x_1^{-1}x_2^{-1}-1)^{-1}$ represent the same noncommutative rational function, in the sense that their evaluations agree on all matrix tuples where 
arithmetic operations in both expressions can be carried out (that is, where the pertinent inverses exist).
Given $\rr\in\rx$, let $\dom_n\rr$ be the domain of $\rr$ in $\mat{n}^m$, defined as the union of the domains of all representatives of $\rr$ (formal rational expressions) in $\mat{n}^m$ \cite{KVV,Vol}.

Let $\gX(n,\xi)$ be an $m$-tuple of $n\times n$ generic matrices, where
$\gX(n,\xi)_k=(\gx_{kij})_{i,j=1}^n \in \opm_n(\kk[\ulg])$ and
$$\ulg=(\gx_{kij}\colon 1\le i,j\le n,\,1\le k\le m )$$
is a tuple of independent commuting indeterminates, viewed as the coordinates of the affine space $\mat{n}^m$. Let $\ud{n}$ be the universal division $\kk$-algebra of degree $n$ generated by $\gX(n,\xi)_1,\dots,\gX(n,\xi)_m$ \cite[Section 3.2]{Row}.

\begin{lem}\label{l:old}
For every $\rr\in\rx$ there exist $\de\in\N$, an affine matrix pencil $L=A_0+A_1x_1+\cdots +A_mx_m$ with $A_j\in\mat{\de}$ and $b,c\in\kk^\de$ such that $\rr=c^\ti L^{-1}b$ holds in $\rx$, and
\begin{equation}\label{e:inc}
\dom_n \rr\supseteq \left\{X\in\mat{n}^m\colon \det L(X)\neq0 \right\}\neq\emptyset
\end{equation}
for all $n\ge \de-1$.
\end{lem}

\begin{proof}
The statement holds by \cite[Corollary 1.3]{CR} and \cite[Propositions 1.12 and 2.10]{DM}.
\end{proof}

Given $\rr\in\rx$, its representation as $\rr=c^\ti L^{-1}b$ as in Lemma \ref{l:old} is called a \emph{linear representation} or \emph{realization} of $\rr$ \cite{CR} of size $\de$. Moreover, the linear representation is \emph{minimal} if $\de$ is minimal; by \cite[Corollary 1.6]{CR}, a minimal linear representation of $\rr$ is unique up to a left-right basis change (and is thus a canonical object representing $\rr$).

\begin{rem}
The scope of Lemma \ref{l:old} is restricted to the purpose of this paper; namely, it ensures the existence of a linear representation of $\rr\in\rx$ (which can be mechanically constructed from a representative of $r$ \cite{CR}) and imposes an effective bound on $n$ for which the invertibility set of $L$ is nonempty \cite{DM}. 
However, Lemma \ref{l:old} may be further strengthened using techniques from control theory. Namely, a minimal linear representation of $\rr$ can also be obtained constructively (albeit with considerably more work) \cite[Section 2]{PV1}, and the inclusion in \eqref{e:inc} is actually an equality when $L$ arises from a minimal linear representation of $\rr$ \cite[Corollary 4.8]{PV2}.
\end{rem}

\begin{rem}
Note that for $f\in\px$, we have that $f$ vanishes (or is singular) on $\mat{n}^m$ if $f$ vanishes (or is singular) on $\mat{n+1}^m$. On the other hand, for $\rr\in\rx$ it can happen that $\rr$ vanishes on $\dom_{n+1}\rr\neq\emptyset$ and $\rr$ is non-singular on $\dom_{n}\rr\neq\emptyset$ \cite{Ber}. Nevertheless, Lemma \ref{l:old} guarantees that this can only happen for small $n$ (relative to the minimal size of a linear representation of $\rr$).
\end{rem}

\begin{lem}\label{l:noncent}
Let $\rr\in\rx\setminus\kk$ have a linear representation of size $\de$. Then for all $n\ge2\de$,
\begin{equation}\label{e:noncent}
\{X\in \dom \rr\colon \rr(X)\notin\kk\cdot I \}
\end{equation}
is a nonempty Zariski open subset of $\mat{n}^m$.

If $\rr\in\px\setminus\kk$ is of degree $d$, then \eqref{e:noncent} is nonempty for all $n\ge\lceil\frac{d}{2}\rceil+1$.
\end{lem}

\begin{proof}
Let $x_0$ be an auxiliary noncommuting variable. By assumption, $(x_0\rr-\rr x_0)^{-1}$ is a well-defined noncommutative rational function in $m+1$ variables. If $\rr=c^\ti L^{-1}b$ is a linear representation of size $\de$, then 
\begin{align*}
(x_0\rr-\rr x_0)^{-1}
&=\left(c^\ti L^{-1} bx_0-c^\ti x_0 L^{-1}b\right)^{-1}\\
&=\begin{pmatrix}1&0&0\end{pmatrix}
\begin{pmatrix}L&0&bx_0 \\ 0 &L & b \\-c^\ti & c^\ti x_0 & 0\end{pmatrix}^{-1}
\begin{pmatrix}1\\0\\0\end{pmatrix}=: \widetilde{c}^\ti \widetilde{L}^{-1}\widetilde{b}
\end{align*}
is a linear representation of size $\widetilde{\de}=2\de+1$, and \eqref{e:noncent} contains the projection of $\dom_n (x_0\rr-\rr x_0)^{-1}\subset \mat{n}^{m+1}$ onto its last $m$ components in $\mat{n}^m$, and
$$\dom_n (x_0\rr-\rr x_0)^{-1}\supseteq \{X\in\mat{n}^{m+1}\colon \det \widetilde{L}(X)\neq0 \}\neq\emptyset$$
for $n\ge \widetilde{\de}-1$ by \cite[Propositions 1.12 and 2.10]{DM}.
Therefore \eqref{e:noncent} is nonempty for $n\ge \widetilde{\de}-1$; it is also evidently Zariski open.

Lastly, if $\rr$ is a nonconstant noncommutative polynomial of degree $d$, then $x_0\rr-\rr x_0$ is a nonzero polynomial on $\mat{n}^m$ for $d+1\le 2n-1$ by \cite[Lemma 1.4.3]{Row}. In particular, \eqref{e:noncent} is nonempty (and clearly Zariski open) for $n\ge \lceil\frac{d}{2}\rceil+1$. 
\end{proof}

\begin{lem}\label{l:noconsteig}
Let $\rr\in\rx$ with $\dom_n\rr\neq\emptyset$ and $\lambda\in\kk$. If $\lambda$ is an eigenvalue of $\rr(X)$ for all $X$ in a Zariski dense subset of $\dom_n\rr$, then $\rr(\gX(n,\xi))=\lambda I$.
\end{lem}

\begin{proof}
By the assumption we have $\det(\rr(\gX(n,\xi))-\lambda I)=0$, so $\rr(\gX(n,\xi))-\lambda I$ is a zero divisor in $\ud{n}$. Since the latter is a division algebra, if follows $\rr(\gX(n,\xi))=\lambda I$.
\end{proof}

Let $R$ be a commutative domain, and $f\in R[t]$. The \emph{discriminant} of $f$ \cite[Section 12.1.B]{GKZ}, $\disc(f)\in R$, is a polynomial in the coefficients of $f$, with the property that $\disc(f)=0$ if and only if $f$ has a repeated root in the algebraic closure  of the quotient field of $R$. This notion is used in the proof of the following is a rational version of \cite[Proposition 1.7]{KB} (see also the proof of \cite[Theorem 4.1]{BS}).

\begin{lem}\label{l:dist_prime}
Let $p\in\N$ be a prime, and let $\rr\in\rx$ with $\dom_p\rr\neq\emptyset$ be such that $\rr(\gX(p,\gx))$ is not central in $\ud{p}$. Then
\begin{equation}\label{e:dist}
	\{X\in \dom_p\rr\colon \rr(X) \text{ has }p \text{ distinct eigenvalues} \}
\end{equation}
is a nonempty Zariski open subset of $\mat{p}^m$.
\end{lem}

\begin{proof}
The set \eqref{e:dist} is clearly Zariski open; one only needs to show that \eqref{e:dist} is nonempty. 
Let $Z_p$ denote the center of $\ud{p}$, and let $F$ be the $Z_p$-subfield of $\ud{p}$ generated by $\rr(\gX(p,\gx))$. Then the inclusions $Z_p\subset F\subset \ud{p}$ are strict; the first one because $\rr(\gX(p,\gx))$ is not central, and the second one because $F$ is commutative and $\ud{p}$ is not. Since $p$ is prime and $\ud{p}$ has dimension $p^2$ over $Z_p$, it follows that $F$ has dimension $p$ over $Z_p$. Let $\chi\in Z_p[t]$ be the characteristic polynomial of $\rr(\gX(p,\gx))$. Then $\deg \chi =p$ and $\chi$ is also the minimal polynomial of $\rr(\gX(p,\gx))$. Furthermore, $\chi$ is irreducible over $Z_p$ since $F\cong Z_p[t]/(\chi)$ is a field. Since $\kk$ has characteristic 0, $Z_p$ is a perfect field, so $\chi$ is separable and has only simple roots in its splitting field. Therefore $\disc(\chi)$ is a nonzero element of $Z_p$. Since $\kk$ is infinite, there exists $X\in\dom_p \rr$ such that $\disc(\chi)|_{\gX(p,\gx)=X}\neq0$. Hence $\chi|_{\gX(p,\gx)=X}$ has $p$ distinct roots, so $\rr(X)$ has $p$ distinct eigenvalues.
\end{proof}

\begin{lem}\label{l:disc}
Let $R,S$ be domains containing $\kk$, and $f\in R[t], g\in S[t]$ without roots in $\kk$. For $fg=f\cdot g\in (R\otimes_{\kk}S)[t]$ we have $\disc(fg)\neq 0$ if and only if $\disc(f)\disc(g)\neq0$. 
\end{lem}

\begin{proof}
Throughout the proof we identify $f=f\otimes 1,g=1\otimes g\in(R\otimes_{\kk}S)[t]$.
By \cite[Section 12.1.B]{GKZ} we have
$$\disc(fg)=(-1)^{\deg f \deg g}\disc(f)\disc(g)\res(f,g)^2,$$
where $\res(f,g)$ is the resultant of $f$ and $g$ \cite[Section 12.1.A]{GKZ}. 
Since $\kk$ is algebraically closed, $R\otimes_{\kk}S$ is a domain (see e.g. \cite[Theorem 1.6 and Example 1.1.10]{Sha}).
Let $F_1,F_2,F$ be algebraic closures of the quotient fields of $R,S,
R\otimes_{\kk}S$, 
respectively; note that $F_1\otimes_{\kk}F_2$ embeds into $F$. The roots of polynomials $f$ and $g$ lie in $(F_1\otimes 1)\setminus\kk$ and $(1\otimes F_2)\setminus\kk$, respectively. In particular, $f$ and $g$ have no common roots, so $\res(f,g)\neq0$ by \cite[Section 12.1.A]{GKZ}.
\end{proof}

The next result shows that nonconstant noncommutative polynomials and rational functions eventually attain values with pairwise distinct eigenvalues (for a stronger conjecture, see \cite[Problem 1.4]{KB}).

\begin{thm}\label{t:dist}
Let $\rr\in\rx\setminus\kk$ admit a linear representation of size $\de$, and let $p<q$ be primes larger than $2\delta$. Then for every $n\ge (p-1)(q-1)$,
\begin{equation}\label{e:dist2}
\{X\in \dom_n\rr\colon \rr(X) \text{ has }n \text{ distinct eigenvalues} \}
\end{equation}
is a nonempty Zariski open subset of $\mat{n}^m$.
\end{thm}

\begin{proof}
By Lemma \ref{l:noncent}, $\rr(\gX(n,\gx))$ is a well-defined non-central element of $\ud{n}$ for every $n\ge 2\de$. Since Zariski openness is evident, it suffices to show that \eqref{e:dist2} is nonempty. The discriminants of the characteristic polynomials of $\rr(\gX(p,\gx))$ and $\rr(\gX(q,\gx))$ are nonzero by Lemma \ref{l:dist_prime}. 
If $n\ge (p-1)(q-1)$, then $n$ can be expressed as $ap+bq$ for some $a,b\in\N$ by Sylvester's theorem \cite[Theorem 2.1.1]{Alf}. 
Let $a,b\in\N$ be arbitrary, and consider the tuple of block diagonal matrices
$$\gX:=\gX(p,\gx^{(1)})\oplus\cdots\oplus \gX(p,\gx^{(a)})\oplus
\gX(q,\gx^{(a+1)})\oplus\cdots\oplus \gX(q,\gx^{(a+b)})$$
over $\kk[\ulg^{(1)}]\otimes_{\kk}\cdots\otimes_{\kk} \kk[\ulg^{(a+b)}]$. Then
$$\rr(\gX)=\rr\left(\gX(p,\gx^{(1)})\right)\oplus\cdots\oplus \rr\left(\gX(p,\gx^{(a)})\right)\oplus
\rr\left(\gX(q,\gx^{(a+1)})\right)\oplus\cdots\oplus \rr\left(\gX(q,\gx^{(a+b)})\right),$$
and the characteristic polynomial of $\rr(\gX)$ is a product of the characteristic polynomials of $\rr\left(\gX(p,\gx^{(i)})\right)$ for $i=1,\dots,a$ and $\rr\left(\gX(q,\gx^{(j)})\right)$ for $j=1,\dots,b$. None of these have a root in $\kk$ by Lemma \ref{l:noconsteig}. Therefore the discriminant of the characteristic polynomial $\chi$ of $\rr(\gX)$ is nonzero by Lemma \ref{l:disc}.
Since $\kk$ is infinite, there exists a tuple of block-diagonal matrices $X\in\dom_{ap+bq} \rr$ such that $\disc(\chi)|_{\gX=X}\neq0$. Hence $\rr(X)$ has $ap+bq$ distinct eigenvalues, so \eqref{e:dist2} is nonempty for $n=ap+bq$.
\end{proof}

\begin{rem}
Using Bertrand's postulate \cite[Theorem 8.7]{NZM}, one concludes that \eqref{e:dist2} is nonempty for all $n\ge (2\de-3-1)(2(2\de-3)-3-1)=4(\de-2)(2\de-5)$.
\end{rem}

\begin{cor} \label{t:dist1}
Let $\rr$ be a nonconstant noncommutative rational function. For all $n$ large enough, the set $\cO_\rr\subseteq\mat{n}^m$ where $\rr$ attains values with distinct and nonzero eigenvalues is Zariski open and nonempty.
\end{cor}

\begin{proof}
Let $n\in\N$ be large enough so that the conclusion of Theorem \ref{t:dist} holds. Then the set $\cO_1\subseteq\mat{n}^m$ where $\rr$ attains values with distinct eigenvalues is Zariski open and nonempty. In particular, $\rr$ is not constantly zero on $\mat{n}^m$; since $\UD_n$ is a division ring, the set $\cO_2\subseteq\mat{n}^m$ where $\rr$ attains invertible values is Zariski open and nonempty. Therefore $\cO_r = \cO_1\cap\cO_2$ is also Zariski open and nonempty. 
\end{proof}

\begin{cor} \label{t:dist2}
Let $f$ be a nonconstant noncommutative polynomial. For all $n$ large enough, the set $\cO_f\subseteq\mat{n}^m$ where $f$ attains values with distinct and nonzero eigenvalues is Zariski open and nonempty.
\\
More precisely, the statement is valid for $n\ge (p-1)(q-1)$ where $p<q$ are primes larger than $\lceil\frac{\deg f}{2}\rceil$.
\end{cor}

\begin{proof}
The improved lower bound on $n$ follows from using the part of Lemma \ref{l:noncent} referring to noncommutative polynomials at the beginning of the proof of Theorem \ref{t:dist}.
\end{proof}

\begin{rem}\label{rrr}
 Recall that a polynomial $f\in\px$ is said to be 
 a polynomial identity of $\mat{n}$
 if $f(X)=0$ for every $X\in \mat{n}^m$.
 If 
 $f(X)$ is always a scalar matrix, but $f$ is not a polynomial identity, then $f$ is said to be 
  central for $\mat{n}$. It is obvious 
  that the conclusion of Corollary \ref{t:dist2} does not hold for polynomial identities and central polynomials. One may wonder whether these are the only polynomials that have to be excluded, and, accordingly, whether the assumption that $n$ must be large enough  is too restrictive. This  is not the case.
Recall further that
$g\in\px$ is $2$-central for $\mat{n}$ if $g^2$ is central for $\mat{n}$, but $g$ is not. For example, $[x_1,x_2]$ is 2-central for 
$\mat{2}$. It is known that $2$-central polynomials for 
$\mat{n}$  exist for various large (even)  $n$ \cite[Proposition 1]{KBMR}. Observe that the evaluations of a  2-central polynomial have at most two distinct eigenvalues.
We may also consider more general $k$-central polynomials, as well sums of these and central  polynomials. 
This justifies the assumption in Corollary \ref{t:dist2}
(and hence also in Theorem \ref{t:dist})
that $n$ must be large enough. On the other hand, this assumption is superfluous for multilinear polynomials \cite[Theorem 3]{KBMR}.  \end{rem}

\section{The linear Waring problem} \label{s3}

We first record an elementary lemma, which was implicitly already noticed in \cite{BS}.

\begin{lem}\label{lnova}
    Let $D\in \mat{n}$ be a matrix with $n$ distinct eigenvalues $\lambda_1,\dots, \lambda_n$.
    Then every matrix in $\slie{n}$ is a difference of two matrices similar  to $D$.
\end{lem}

\begin{proof} 
By \cite[Section 2.2, Problem 3]{HJ}, every matrix in $M\in \slie{n}$  is similar to a matrix all of whose diagonal entries are zero. Therefore $M=Q(N_1-N_2)Q^{-1}$ for some $Q\in\GL{n}$, an upper-triangular matrix $N_1$ with diagonal entries $\lambda_1,\dots,\lambda_n$, and a lower-triangular matrix $N_2$ with diagonal entries $\lambda_1,\dots,\lambda_n$. Observe that $QN_1Q^{-1}$ and $QN_2Q^{-1}$ are similar to $D$.
\end{proof}

Our first application of Theorem \ref{t:dist} is the following theorem.

\begin{thm}\label{c:diff}
For every nonconstant $\rr\in\rx$ and large enough $n\in\N$,
$$\slie{n}\subseteq \rr(\dom_n \rr)-\rr(\dom_n \rr).$$
\end{thm}

\begin{proof}
Given $\rr$, let $n\in\N$ be large enough so that the conclusion of Theorem \ref{t:dist} holds; that is, there exists $X\in\dom_n\rr$ such that $\rr(X)$ has $n$ distinct eigenvalues $\lambda_1,\dots,\lambda_n$.
Take $M\in\slie{n}$.
By Lemma \ref{lnova}, there exist $P_1,P_2\in\GL{n}$ such that
$M=P_1\rr(X)P_1^{-1}-P_2\rr(X)P_2^{-1}$, and hence
$M=\rr(P_1\ XP_1^{-1})-\rr(P_2 XP_2^{-1})$. \end{proof}

As mentioned in the introduction, the special case  of Theorem 
\ref{c:diff}  where $r=f \in\px$ gives a positive answer to  \cite[Question 4.8]{BS}. Observe that if $f$ is a sum of commutators in $\px$, then all matrix evaluations
of $f$ consist of trace zero matrices. 
This justifies the restriction to $\slie{n}$ in the conclusion of Theorem 
\ref{c:diff}. It is natural to ask what can be said if $f$ is not a sum of commutators. To give a (partial) answer, we first record the following lemma.

\begin{lem} \label{led} Let $\lambda_1,\dots,\lambda_n$ be distinct elements in $\mathbb K$ having nonzero sum.
 Then every nonscalar matrix $A\in \mat{n}$ with nonzero trace
can be written as $A= \alpha(D_1 +  D_2)$, where $\alpha\in \mathbb K$ and $D_1,D_2$ are matrices 
similar to ${\rm diag}(\lambda_1,\dots,\lambda_n)$.
\end{lem}

\begin{proof} 
 Let $ \alpha \in\mathbb K$ be such that $2\alpha(\lambda_1+\dots +\lambda_n)$ is equal to the trace of $A$.  By \cite[Lemma 2.8]{BSb},  there exists a matrix which is similar to $A$ and whose diagonal entries are $2\alpha \lambda_1,\dots,  2\alpha \lambda_n$. 
Therefore, there is no loss of generality in assuming that 
$2\alpha \lambda_1,\dots, 2 \alpha \lambda_n$
 are the diagonal entries of $A$. 
Write $$A=  {\rm diag}(2\alpha \lambda_1,\dots,  2\alpha \lambda_n) + U+ L,$$ where $U$ is a strictly upper triangular matrix and $L$ is a strictly lower diagonal matrix. We remark that $\alpha\ne 0$ since $A$ has nonzero trace. Note that the matrices 
$$D_1= {\rm diag}(\lambda_1,\dots,\lambda_n) + \alpha^{-1} U,\quad D_2= {\rm diag}(\lambda_1,\dots,\lambda_n) + \alpha^{-1}L$$ are similar to ${\rm diag}(\lambda_1,\dots,\lambda_n)$ and $A= \alpha(D_1 +  D_2)$.
\end{proof}

\begin{rem} The assumption that  $A$ is nonscalar is necessary.  In fact, the identity
$I$  is not always a linear combination of two matrices similar to ${\rm diag}(\lambda_1,\dots,\lambda_n)$.
Indeed, $I=\beta D_1 + \gamma D_2$ implies that the  spectrum of $I-\beta D_1$ is equal to the spectrum of $\gamma D_2$, that is,
$\{1-\beta \lambda_1,\dots, 1-\beta \lambda_n\}=\{\gamma \lambda_1,\dots, \gamma\lambda_n\}$.
It is easy to find $\lambda_i$'s for which there are  no such $\beta,\gamma$.
\end{rem}

Let $f\in\px$.
We simplify the notation and write
$f(\mat{n})$ for $f(\mat{n}^m)$.
More generally, 
for any algebra $\mathcal A$, we   
 write $f(\mathcal A)$ for $f(\mathcal A^m)$. We call
$f(\mathcal A)$  the image of $f$ in $\mathcal A$.
Note that $f(\mathcal A)$ is closed under conjugation by invertible elements in $\mathcal A$.

\begin{thm}\label{c23}
If a polynomial $f$ is not a sum of commutators, then, for $n$ large enough, every nonscalar matrix in $\mat{n}$ is a linear combination of two matrices from $f(\mat{n})$.
\end{thm}

\begin{proof} 
We claim that for $n$ large enough, $f(\mat{n})$ contains a matrix with nonzero trace. 
Since $f$ is not a sum of commutators,  \cite[Corollary 4.7]{BK} tells us that there exists a natural number $n_0$ such that
$f(\mat{n_0})$ contains a matrix with nonzero trace. If $f$ has constant term zero,  this readily implies that
the same is true for $f(\mat{n})$ for every $n\ge n_0$. If, however, $f$ has a nonzero constant term, then obviously 
 $f(0,\dots,0)$ has  nonzero trace for each $n$.  

Thus, for $n$ large enough, the tuples in $\mat{n}^m$ in which $f$ attains a matrix with nonzero trace form  a nonempty Zariski open subset of $\mat{n}^m$. Corollary \ref{t:dist2} shows that for $n$ large enough, the  tuples in $\mat{n}^m$ in which $f$ attains a matrix with $n$ distinct eigenvalues also form  a nonempty Zariski open subset of $\mat{n}^m$. Their intersection is therefore nonempty. Thus, for $n$ large enough, $f(\mat{n})$ contains a matrix $D$
with nonzero trace and $n$ distinct eigenvalues.
Since $f(\mat{n})$ is closed under conjugation with invertible matrices, we may assume that $D$ is a diagonal matrix.
  Lemma \ref{led} therefore shows that every nonscalar matrix 
in $\mat{n}$ having nonzero trace is a scalar multiple of the sum of two matrices similar to $D$. By Lemma \ref{lnova}, every trace zero matrix in $\mat{n}$ is a difference of two matrices similar to $D$.
Since matrices similar to $D$ belong to $f(\mat{n})$, this proves the corollary.
\end{proof}

\begin{rem}
Taking a polynomial like $f=1+[x,y]$ we see that one really has to involve linear combinations, i.e., sums and differences are insufficient.\end{rem}

 If $f$ attains a nonzero value in scalars (i.e., $f(\alpha_1,\dots,\alpha_m)\ne 0$ for some $\alpha_j\in\mathbb K$), in particular if $f$ has nonzero constant term,
 then we can omit ``nonscalar" in the statement of Theorem \ref{c23}. We 
 leave the general case as an open problem.
 
\begin{question}\label{q36}
    Does the conclusion of Theorem \ref{c23} also hold for scalar matrices?
\end{question}

 What can be said is the following.

\begin{cor}\label{c23a}
If a polynomial $f$ is not a sum of commutators, then, for $n$ large enough, every  matrix in $\mat{n}$ is a linear combination of three matrices from $f(\mat{n})$.
\end{cor}

\begin{proof}
    As noticed at the beginning of the proof of Theorem \ref{c23}, $f(\mat{n})$ contains a matrix  $A$ with nonzero trace. Since every matrix in $\mat{n}$ is a linear combination of $A$ and a traceless matrix, the desired conclusion follows from  Theorem \ref{c:diff}.
\end{proof}

In view of Theorem \ref{c:diff}, Question \ref{q36} can be equivalently asked as follows: Given any nonconstant polynomial $f$, is the set of linear combinations of two matrices from $f(\mat{n})$ a vector space for all large enough $n$? We remark that 
\cite[Theorem 1.2]{BS} shows that the restriction to large enough $n$ is necessary here.

In our last theorem of this section we will consider algebras more general than $\mat{n}$.
Here, by an algebra we mean a unital, associative and not necessarily commutative algebra over $\mathbb K$.
We first record an elementary lemma.

\begin{lem} \label{l31}Let $\mathcal A$ be an algebra, let $\lambda_1,\dots,\lambda_m$ be distinct elements in $\mathbb K$,
and let $d={\rm diag}(\lambda_1 1,\dots,\lambda_m 1)\in \opm_m(\mathcal A)$, where  $1$ stands for the unity of $\mathcal A$. If $b$ is an upper triangular  (resp.\ lower triangular) matrix in $\opm_m(\mathcal A)$  with the same diagonal as $d$, then there exists an upper triangular (resp.\ lower triangular) matrix $t\in \opm_m(\mathcal A)$ with $1$'s on the diagonal such that $b=t d t^{-1}$. 
\end{lem}

The lemma can be easily proved  by induction on $m$ or directly. We leave the details to the reader.

Recall that a matrix over a field has trace zero if and only if it is a commutator of two matrices \cite{AM}. When considering 
matrices over algebras that are not fields, 
the role of trace zero matrices are therefore played by commutators.

\begin{thm}\label{cbx} Let $\mathcal A$ be an algebra over $\mathbb K$ such that $\mathcal A\cong \opm_2(\mathcal A)$.  If $f$ is a nonconstant polynomial, then every commutator in $\mathcal A$ is a sum of seven elements from $f(\mathcal A )-f(\mathcal A)$.
\end{thm}
\begin{proof} Observe that $\mathcal A\cong \opm_2(\mathcal A)$ implies that $\mathcal A\cong \opm_{2 ^n}(\mathcal A)$ for every positive integer $n$. By Corollary  \ref{t:dist2}, there exists a positive integer $n$ such that $f(\mat{2^n}) $ contains a diagonal matrix  diag$(\lambda_1,\dots,\lambda_{2^n})$ with distinct $\lambda_i$. In light of the natural embedding of $\mat{2^n}$ into 
$\opm_{2^n}(\mathcal A)$, this implies that $f(\opm_{2 ^n}(\mathcal A))$  contains the diagonal matrix
 $$d={\rm diag}(\lambda_1 1,\dots,\lambda_{2^n} 1)\in \opm_{2 ^n}(\mathcal A).$$

 Since $f(\opm_{2 ^n}(\mathcal A))$ is closed under conjugation with invertible elements, Lemma \ref{l31} shows that $f(\opm_{2 ^n}(\mathcal A))$ contains all matrices of the form $d + u$  where $u$ is  a strictly upper triangular matrix in $\opm_{2 ^n}(\mathcal A)$, as well as all matrices of the form $d+\ell$ where $\ell$ is a  strictly lower triangular matrix in $\opm_{2 ^n}(\mathcal A)$. Therefore, $$u-\ell\in f(\opm_{2 ^n}(\mathcal A))-f(\opm_{2 ^n}(\mathcal A)),$$ meaning that $f(\opm_{2 ^n}(\mathcal A))-f(\opm_{2 ^n}(\mathcal A))$ contains all matrices in  $\opm_{2 ^n}(\mathcal A)$ that have zeros on the diagonal. In particular, all matrices $\left[ \begin{smallmatrix} 0 & a  \cr b&0 \cr \end{smallmatrix} \right]$ with
$a,b\in \opm_{2 ^{n-1}}(\mathcal A)$ are  contained in $f(\opm_{2 ^n}(\mathcal A))-f(\opm_{2 ^n}(\mathcal A))$. 

In view of our assumption on $\mathcal A$, we may replace   $\opm_{2 ^{n-1}}(\mathcal A)$ by $\mathcal A$ in the last statement.  Thus,
$f(\opm_2(\mathcal A)) - f(\opm_2(\mathcal A))$ contains the subspace $\mathcal S$ of $\opm_2(\mathcal A)$ consisting of matrices $\left[ \begin{smallmatrix} 0 & a  \cr b&0 \cr \end{smallmatrix} \right]$ with $a,b\in \mathcal A$. 
Take  a commutator $k\in \opm_2(\mathcal A)$. We have
\begin{align*} k=&\left[ \begin{matrix} a & b  \cr c&d \cr \end{matrix} \right]\left[ \begin{matrix} e & f  \cr g&h \cr \end{matrix} \right] - \left[ \begin{matrix} e & f  \cr g&h \cr \end{matrix} \right] \left[ \begin{matrix} a & b  \cr c&d \cr \end{matrix} \right]\\
=& \left[ \begin{matrix} ae - ea + bg  - fc & 0 \cr 0& dh-hd -gb +cf\cr \end{matrix} \right] + s\end{align*}
where $s\in \mathcal S$.
Using that 
$\left[ \begin{smallmatrix} 1 & x  \cr 0&1 \cr \end{smallmatrix} \right]^{-1} =\left[ \begin{smallmatrix} 1 & -x \cr 0&1 \cr \end{smallmatrix} \right]$ for every $x\in \mathcal A$, it follows that
\begin{align*} k
=& \left[ \begin{matrix} 1 & 1 \cr 0&1 \cr \end{matrix} \right] \left[ \begin{matrix} 0 & 0  \cr ae&0 \cr \end{matrix}
 \right] \left[ \begin{matrix} 1 & 1 \cr 0&1 \cr \end{matrix} \right]^{-1} + 
\left[ \begin{matrix} 1 & -e  \cr 0&1 \cr \end{matrix} \right] \left[ \begin{matrix} 0 & 0  \cr a&0 \cr \end{matrix} \right] \left[ \begin{matrix} 1 & -e  \cr 0&1 \cr \end{matrix} \right]^{-1} \\
+ & 
\left[ \begin{matrix} 1 & b  \cr 0&1 \cr \end{matrix} \right] \left[ \begin{matrix} 0 & 0  \cr g&0 \cr \end{matrix} \right] \left[ \begin{matrix} 1 & b  \cr 0&1 \cr \end{matrix} \right]^{-1} +
\left[ \begin{matrix} 1 & -f  \cr 0&1 \cr \end{matrix} \right] \left[ \begin{matrix} 0 & 0  \cr c&0 \cr \end{matrix} \right] \left[ \begin{matrix} 1 & -f  \cr 0&1 \cr \end{matrix} \right]^{-1}\\
+ & \left[ \begin{matrix} 1 & -1 \cr 0&1 \cr \end{matrix} \right] \left[ \begin{matrix} 0 & 0  \cr dh&0 \cr \end{matrix}
 \right] \left[ \begin{matrix} 1 & -1 \cr 0&1 \cr \end{matrix} \right]^{-1} + 
\left[ \begin{matrix} 1 & d  \cr 0&1 \cr \end{matrix} \right] \left[ \begin{matrix} 0 & 0  \cr h&0 \cr \end{matrix} \right] \left[ \begin{matrix} 1 & d  \cr 0&1 \cr \end{matrix} \right]^{-1} + s'
\end{align*}
where $s'\in \mathcal S$. Each of the first six terms is a conjugate of a matrix from $\mathcal S$. As $\mathcal S \subseteq f(\opm_2(\mathcal A)) - f(\opm_2(\mathcal A))$ and $f(\opm_2(\mathcal A)) - f(\opm_2(\mathcal A))$ is closed under conjugation, it follows that $k$ is a sum of seven elements from $f(\opm_2(\mathcal A)) - f(\opm_2(\mathcal A))$. Since $\mathcal A\cong \opm_2(\mathcal A)$, this completes the proof.
\end{proof} 

We  point out two cases to which Theorem \ref{cbx} applies. 
The first corollary is a version of \cite[Corollary 3.15]{B} which imposes milder conditions on $\mathbb K$, but in which  the role of seven is played by a certain four-digit number.

\begin{cor}\label{cbx2} Let $V$ be an infinite-dimensional vector space over $\mathbb K$. If $f$ is a nonconstant polynomial, then every element in $\mathcal A={\rm End}_F(V)$ is a sum of seven elements from $f(\mathcal A )-f(\mathcal A)$.
\end{cor}

\begin{proof}
From  $V \cong V\oplus V$ we obtain $\mathcal A\cong \opm_2(\mathcal A)$. Since every element in $\mathcal A$ is a commutator \cite[Proposition 12]{M},
the desired conclusion follows from Theorem \ref{cbx}.
\end{proof}

By $B(X)$ we denote the algebra of all continuous linear operators on a Banach space $X$.
If $X\cong X\oplus X$, then $\mathcal A=B(X)$ satisfies $\mathcal A\cong \opm_2(\mathcal A)$. Thus the following corollary holds.

\begin{cor}\label{cbx3} Let $X$ be a
Banach space such that $X\cong X\oplus X$. If $f$ is a nonconstant polynomial, then every commutator in $\mathcal A=B(X)$ is a sum of seven elements from $f(\mathcal A)-f(\mathcal A)$.
\end{cor}

Infinite-dimensional  Hilbert spaces $H$ certainly satisfy  $H\cong H\oplus H$, but for them Corollary \ref{cbx3} is not entirely new; see \cite[Corollary 3.17]{B} and \cite[Theorem 4.2]{BS} (since every operator in $B(H)$ is a sum of two commutators \cite{H}, the conclusions of these two results  are not restricted to commutators).  Corollary \ref{cbx3} may thus be viewed as a generalization from Hilbert spaces to Banach spaces $X$ isomorphic to their square  $X\oplus X$. This actually covers most Banach spaces. In fact, the first infinite-dimensional Banach space not isomorphic to its square was discovered only in 1960 \cite{BP}, thereby answering a problem posed by Banach in 1932.

It is unlikely that the number seven from the last three results is the least possible. 

\begin{question}
What is the smallest number that can replace the number seven in Theorem \ref{cbx} and Corollaries \ref{cbx2}
and \ref{cbx3}?
\end{question}
 
\section{The mutiplicative Waring problem}\label{s4}

We  write $\kk^\times$ for the group of all nonzero scalar  matrices. Our first result in this section is the multiplicative analog of
Theorem \ref{c:diff}.

\begin{thm}\label{c:quot}
For every nonconstant $\rr\in\rx$ and large enough $n\in\N$,
$$\SL{n}\setminus \kk^\times \subseteq \rr(\dom_n\rr)\cdot \rr(\dom_n\rr)^{-1}.$$
\end{thm}

\begin{proof}
Let $n\in\N$ be large enough so that the conclusion of Corollary \ref{t:dist1} holds.  Then there exists $X\in\dom_n\rr$ such that $\rr(X)$ has $n$ distinct and nonzero eigenvalues $\lambda_1,\dots,\lambda_n$.

Take $M\in\SL{n}\setminus \kk^\times$. By \cite[Theorem 1]{Sou} there exists matrices $N_1,N_2\in\GL{n}$ with eigenvalues $\lambda_1,\dots,\lambda_n$ such that $M=N_1N_2^{-1}$. For $j=1,2$ there exists $P_j\in\GL{n}$ such that $N_j=P_j\rr(X)P_j^{-1}$. Then $M= \rr(P_1XP_1^{-1})\rr(P_2XP_2^{-1})^{-1}$.
\end{proof}

We do not know whether the exclusion of $\kk^\times$ is necessary.

\begin{question}
    Is (in Theorem \ref{c:quot}) 
    $\SL{n}\cap \kk^\times $ also contained in $ \rr(\dom_n\rr)\cdot \rr(\dom_n\rr)^{-1}$?
\end{question}

Since the identity matrix is trivially contained in
$ \rr(\dom_n\rr)\cdot \rr(\dom_n\rr)^{-1}$, the following corollary to Theorem \ref{c:quot} holds.

\begin{cor}\label{c:quot2}
For every nonconstant $\rr\in\rx$ and large enough $n\in\N$,
$$\GL{n} \subseteq \kk^\times \cdot \rr(\dom_n\rr)\cdot \rr(\dom_n\rr)^{-1}.$$
\end{cor}

In the rest of the section we consider the multiplicative Waring problem for noncommutative polynomials. We start with an auxiliary result.

\begin{prop}\label{p:surjdet}
Let $f\in\px$ and $n\in\N$ be such that $f$ is not constant on $\mat{n}^m$. Then $\det f(\mat{n})=\kk$.
\end{prop}

\begin{proof}
It suffices to see that $\det f(\Xi(n,\xi))\in\kk[\ulg]$ is a nonconstant polynomial. Suppose $\det f(\Xi(n,\xi))=\alpha\in\kk$. By the Cayley-Hamilton theorem, $f(\Xi(n,\xi))\cdot p=\alpha$ for a trace polynomial $p$. The algebra of trace polynomials \cite[Section 2]{Pro} is graded by degree, and has no zero divisors. Let $\tilde{f}$ and $\tilde{p}$ be the homogeneous components of $f(\Xi(n,\xi))$ and $p$, respectively, of highest degree. Since $f$ is not constant on $\mat{n}^m$, we have $\deg \tilde{f}>0$. Then $f(\Xi(n,\xi))\cdot p=\alpha$ implies $\tilde{f}\cdot \tilde{p}=0$, and so $\tilde{p}=0$. Hence $p=0$, and therefore $\alpha=0$. Thus $f(\Xi(n,\xi))$ is not invertible in $\UD_n$, which is a division ring. Hence $f(\Xi(n,\xi))=0$, contradicting the assumption that $f$ is not constant.
\end{proof}

\begin{exa}\label{ex:ratdet}
Let $\rr=x_1x_2x_1^{-1}x_2^{-1}\in\rx$. 
Then $\rr$ is nonconstant (and in particular attains values with pairwise distinct eigenvalues by Theorem \ref{t:dist}), but $\det \rr(X)=1$ for every $X\in\dom_n \rr$ and $n\in\N$.
\end{exa}

\begin{thm}\label{c:prod}
For all nonconstant $f,g\in\px$ and large enough $n\in\N$,
$$\GL{n}\setminus\kk^\times \subseteq f(\mat{n})\cdot g(\mat{n}).$$
\end{thm}

\begin{proof}
Let $n\in\N$ be large enough so that the conclusion of Corollary \ref{t:dist2} holds for both $f$ and $g$.
By Corollary \ref{t:dist1}, 
  the set $\cO_f\subseteq\mat{n}^m$ where $f$ attains values with distinct and nonzero eigenvalues is Zariski open and nonempty.
The set $\det f(\cO_f)$ is constructible in $\kk$ by Chevalley's theorem \cite[Theorem 3.16]{Har}. In particular, $\det f(\cO_f)$ is either finite or cofinite in $\kk$. The first case would contradict $\det f(\mat{n})=\kk$ which holds by by Proposition \ref{p:surjdet}, so $\det f(\cO_f)$ is cofinite in $\kk$. Of course, for an analogously defined set $\cO_g$,   $\det g(\cO_g)$ is also cofinite in $\kk$.

Take $M\in \GL{n}\setminus\kk^\times$. Since $\kk\setminus \det f(\cO_f)$ and $\kk\setminus \det g(\cO_g)$ are finite, there exist $X_1,X_2\in\cO$ such that $\det M=\det f(X_1)\det g(X_2)$. By \cite[Theorem 1]{Sou}, there are matrices $N_1,N_2\in\GL{n}$ such that $M=N_1N_2$, eigenvalues of $N_1$ coincide with those of $f(X_1)$, and eigenvalues of $N_2$ coincide with those of $g(X_2)$. Therefore $M=f(P_1XP_1^{-1})g(P_2XP_2^{-1})$ for suitable $P_1,P_2\in\GL{n}$.
\end{proof}

\begin{question}
   Is (in Theorem \ref{c:prod}) 
    $ \kk^\times $ also contained in $ f(\mat{n})\cdot g(\mat{n})$?
\end{question}

At least we know that three factors  are sufficient for covering the whole $\GL{n}$.

\begin{cor}\label{c:prod2}
For all nonconstant $f,g,h\in\px$ and large enough $n\in\N$,
$$\GL{n} \subseteq f(\mat{n})\cdot g(\mat{n})\cdot h(\mat{n}).$$
\end{cor}

\begin{proof} Let $n$ be large enough so that  the conclusion of Theorem \ref{c:prod} holds for $f$ and $g$  as well as 
the conclusion of Corollary \ref{t:dist1}
holds for $h$. Let $A_1\in h(\mat{n})$ be a matrix that has $n$ distinct and nonzero eigenvalues, and let
$A_2$ be a matrix that is similar to but linearly independent with $A_1$.  Note that $A_2$
also lies in $h(\mat{n})$.   By Theorem \ref{c:prod}, $f(\mat{n})\cdot g(\mat{n})\cdot A_i$ contains all matrices from $\GL{n}$ besides possibly scalar multiples of $A_i$, $i=1,2$. This readily implies the desired conclusion. 
\end{proof}

It would be desirable to also cover singular matrices by products of the images of polynomials. In our final result we will obtain a result of such kind, however, by restricting to polynomials that have a root in scalars. By this we mean that there exist  $\alpha_j\in\mathbb K$ such that
$f(\alpha_1,\dots,\alpha_m)=0$. This obviously holds if  $f$ has zero constant term.

\begin{thm}\label{cza}
If  a nonzero polynomial $f$  has a root in scalars, then   for $n$ large enough, every matrix in
$\mat{n}$ is a product of twelve matrices from $f(\mat{n})$.
\end{thm}

\begin{proof} By Corollary \ref{c:prod2}, there exists an
  $N_0\ge 1$ such that for every $n\ge N_0$,  every invertible matrix in $\mat{n}$ is a product of three matrices from $f(\mat{n})$.
Set $N=2N_0$ and take $n\ge N$. Since every matrix in $\mat{n}$ is a product of two diagonalizable matrices \cite{Botha}, it is enough to show that
every  diagonalizable matrix  $D\in \mat{n}$ is a product of six matrices from $f(\mat{n})$.

 Write $k$ for the rank of $D$. 
Since $f(\mat{n})$  is closed under conjugation by invertible matrices, we may assume  that  $$D = {\rm diag}(\lambda_1,\dots,\lambda_k,0,\dots,0)$$ where
$\lambda_i$ are  nonzero elements in  $\mathbb K$  (so there are $n-k$ zeros).

Assume first that $k\ge N_0$. We claim that in this case $D$ is actually a product of three  matrices from $f(\mat{n})$ (since this in particular implies that the identity matrix $I$ is a product of three  matrices from $f(\mat{n})$, $D=DI$ is  then a product of 
six matrices from $f(\mat{n})$, as desired). 
 As we know, our assumption on $N_0$ implies that ${\rm diag}(\lambda_1,\dots,\lambda_k)$ is a product of three matrices from $f(\mat{k})$:   $${\rm diag}(\lambda_1,\dots,\lambda_k) = f(A_{11},\dots,A_{1m})\cdot f(A_{21},\dots,A_{2m})\cdot f(A_{31},\dots,A_{3m})$$
with $A_{ij}\in \mat{k}$. 
Let $\alpha_i\in\mathbb K$ be such that
$f(\alpha_1,\dots,\alpha_m)=0$.  
Setting
$$
A_{ij}' = \left[ \begin{matrix} A_{ij} & 0  \cr 0&\alpha_j I_{n-k} \cr \end{matrix} \right]\in \mat{n},
$$
where $I_{n-k}$ is the $(n-k)\times (n-k)$ identity matrix and 
 the zeros are blocks of the appropriate size,
we have 
$$D=  f(A_{11}',\dots,A_{1m}')\cdot f(A_{21}',\dots,A_{2m}')\cdot f(A_{31}',\dots,A_{3m}'),$$
proving our claim.

Now consider the  case where  $k< N_0$.  Set
$$D_1 = {\rm diag}(\lambda_1,\dots,\lambda_k, 1,\dots,1, 0,\dots,0)$$
where  $1$ occurs  $N_0-k$  times and $0$ occurs $n-N_0$ times, and 
$$D_2 = {\rm diag}(1,\dots,  1, 0,\dots,0, 1,\dots,  1 )$$
where $1$ occurs $k$ times in the first sequence,  $0$ occurs $N_0-k$ times, 
and  $1$ occurs $n-N_0$ times in the last sequence. 
Since $D_1$ and $D_2$ are diagonal matrices of rank $N_0$ and $n-N_0+k\ge N_0$, respectively, each of them is a product of three matrices from $f(\mat{n})$. Hence, $D=D_1D_2$ is a product of six matrices from $f(\mat{n})$.
\end{proof}

\begin{question}
What is the smallest number that can replace the number twelve in  Theorem \ref{cza}?
\end{question}

\begin{rem}
	Observe that if $g$ is 2-central for $\mat{n}$, then $f=g^3$ satisfies
	$f(\mat{n})\subseteq \GL{n}\cup\{0\}.$ This may be viewed as a justification for the assumption in Theorem \ref{cza} that $n$ must be large enough (compare Remark \ref{rrr}). 
\end{rem}

\begin{question}\label{q:twelve}
	Does Theorem \ref{cza} hold for all nonconstant polynomials $f$?
\end{question}

\begin{rem}
Question \ref{q:twelve} is related to the following one by Leonid Makar-Limanov: \emph{
	Given a nonconstant $f\in\px$ let $\rho_n=\min\rk f(\mat{n})$ for $n\in\N$; is it true that $\liminf_n  \frac{\rho_n}{n}=0$?
} As observed by Makar-Limanov, an affirmative answer to this question implies that the principal ideal in $\px$ generated by $f$ is proper. The latter seemingly innocuous statement turns to be always true if $\kk$ has characteristic 0 (by the existence of an algebraically closed skew field constructed by Makar-Limanov \cite{makarlimanov}), but remains open for $\kk$ with positive characteristic.
On the other hand, even in characteristic 0, Makar-Limanov's question is compelling to the free analysis community, as its affirmative answer would mitigate the fact that not every nonconstant noncommutative polynomial has a matrix root.

As the proof of Theorem \ref{cza} indicates, a crucial obstruction to answering Question \ref{q:twelve} is a lack of grasp on low-rank values of noncommutative polynomials. To inspire further work in this direction, we pose the following refinement of  Makar-Limanov's question. 

\begin{question}\label{q:last}
If $f\in\px$ is not constant on $n\times n$ matrices, does there exist $X\in\mat{n}^m$ such that $\rk f(X)\le 1$?
\end{question}

\end{rem}


\begin{thebibliography}{KK}




\bibitem[AM57]{AM}A. Albert, B. Muckenhoupt, {\it On matrices of trace 0},  Michigan Math. J.  1 (1957) 1--3.


\bibitem[Alf05]{Alf}
J. L. R. Alfons{\'i}n:
{\it The diophantine Frobenius problem},
Oxford University Press, Oxford, 2005.

\bibitem[Ami66]{Ami}
S. A. Amitsur:
{\it Rational identities and applications to algebra and geometry},
J. Algebra 3 (1966) 304--359.

\bibitem[AHKM18]{AHKM}
M. Augat, J.W. Helton, I. Klep, S. McCullough: 
{\it Bianalytic Maps Between Free Spectrahedra},
Math. Ann. 371 (2018) 883--959.

\bibitem[Ber76]{Ber}
G. M. Bergman:
{\it Rational relations and rational identities in division rings. I},
J. Algebra 43 (1976) 252--266.

\bibitem[BR11]{BR} J. Berstel, C. Reutenauer: 
{\it Noncommutative rational series with applications},
Encyclopedia of Mathematics and its Applications 137, Cambridge University Press, Cambridge, 2011.

\bibitem[BP60]{BP}
C.\ Bessaga,  A.\ Pe\l czy\' nski: {\it
Banach spaces non-isomorphic to their Cartesian squares}, Bull. Acad. Polon. Sci. S\' er. Sci. Math. Astr. Phys. 8 (1960) 77--80.

\bibitem[Bot98]{Botha} J. D. Botha: {\it Products of diagonalizable matrices}, Linear Algebra Appl. 273 (1998) 65--82.



\bibitem[Bre20]{B}M.\ Bre\v sar: {\it  Commutators and images of noncommutative polynomials},  Adv. Math.  374 (2020) 107346, 21 pp.

\bibitem[BK09]{BK}M. Bre\v sar, I. Klep: {\it
Values of noncommutative polynomials, Lie skew-ideals and tracial Nullstellens\" atze}, 
Math. Res. Lett. 16 (2009) 605--626.

\bibitem[B\v{S}23a]{BS}
M. Bre\v{s}ar, P. \v{S}emrl:
{\it The Waring problem for matrix algebras},
Israel J. Math. 253 (2023) 381--405.



\bibitem[B\v{S}23b]{BSb} M. Bre\v sar, P. \v Semrl: {\it The Waring problem for matrix algebras}, II, Bull. London Math. Soc. 55 (2023) 1880--1889. 

\bibitem[Che23]{Chen} Q. Chen, {\it On Panja-Prasad conjecture},
arXiv: 2306.15118.

\bibitem[Coh06]{Coh}
P. M. Cohn:
{\it Free Ideal Rings and Localization in General Rings},
New Mathematical Monographs 3, Cambridge University Press, Cambridge, 2006.

\bibitem[CR99]{CR}
P. M. Cohn, C. Reutenauer: 
{\it On the construction of the free field},
Internat. J. Algebra Comput. 9 (1999) 307--323.

\bibitem[CMMPY22]{CMMPY}
B. Collins, B., T. Mai, A. Miyagawa, F. Parraud, S. Yin:
{\it Convergence for noncommutative rational functions evaluated in random matrices}, 
Math. Ann. (2022) DOI:10.1007/s00208-022-02530-5.

\bibitem[DM17]{DM}
H. Derksen, V. Makam:
{\it Polynomial degree bounds for matrix semi-invariants},
Adv. Math. 310 (2017) 44--63.

\bibitem[GGRW05]{GGRW}
I. M. Gelfand, S. Gelfand, V. Retakh, R. L. Wilson: {\it Quasideterminants},
Adv. Math. 193 (2005) 56--141.

\bibitem[GKZ94]{GKZ}
I. M. Gelfand, M. M. Kapranov, A. V. Zelevinsky:
{\it Discriminants, resultants, and multidimensional determinants},
Birkh\"auser Boston, 1994.

\bibitem[Hal54]{H}
P. R. Halmos: {\it Commutators of operators II}, Amer. J. Math. 76 (1954) 191--198.

\bibitem[Har92]{Har}
J. Harris:
{\it Algebraic geometry: a first course},
Graduate Texts in Mathematics 133, Springer-Verlag, New York, 1992.

\bibitem[HMS18]{HMS}
J. W. Helton, T. Mai, R. Speicher:
{\it Applications of realizations (aka linearizations) to free
probability},
J. Funct. Anal. 274 (2018) 1--79.

\bibitem[HMV06]{HMV}
J. W. Helton, S. McCullough, V. Vinnikov: 
{\it Noncommutative convexity arises from linear
matrix inequalities},
J. Funct. Anal. 240 (2006) 105--191.

\bibitem[HJ85]{HJ}
R. A. Horn, C. R. Johnson:
{\it Matrix analysis},
Cambridge University Press, London, 1985.

\bibitem[K-VV12]{KVV}
D. S. Kaliuzhnyi-Verbovetskyi, V. Vinnikov:
{\it Noncommutative rational functions, their difference-differential calculus and realizations}, 
Multidimens. Syst. Signal Process. 23 (2012) 49--77.

\bibitem[K-VV14]{KVV1}
D. S. Kalyuzhnyi-Verbovetskyi, V. Vinnikov:
{\it Foundations of free noncommutative function theory}, 
Mathematical Surveys and Monographs 199, American Mathematical Society, Providence
RI, 2014.

\bibitem[K-BMR16]{KBMR}
A.Kanel-Belov, S.Malev, L.Rowen, {\it Power-central polynomials on matrices}, J. Pure Appl. Algebra 220 (2016) 2164--2176.

\bibitem[K-BRZ23]{KB}
A. Kanel-Belov, F. Razavinia, W. Zhang: 
{\it Centralizers in free associative algebras and generic matrices},
Mediterr. J. Math. 20 (2023) no 85.

\bibitem[LST11]{LST}M.\,Larsen, A.\,Shalev, P.\,H.\,Tiep: {\it 
The Waring problem for finite simple groups},
Ann.  Math.  174 (2011)  1885--1950.

\bibitem[M-L]{makarlimanov}
L. Makar-Limanov:
{\it Algebraically closed skew fields},
J. Algebra 93 (1985) 117--135.

\bibitem[Mes]{M} 
Z. Mesyan: {\it Commutator rings}, Bull. Austral. Math. Soc. 74 (2006) 279--288.

\bibitem[NZM91]{NZM}
I. Niven, H. S.Zuckerman, H. L. Montgomery:
{\it An introduction to the theory of numbers},
John Wiley \& Sons, New York, 1991.

\bibitem[PP23]{PP}
S. Panja, S. Prasad, {\it 
The image of polynomials and Waring type problems on upper triangular matrix algebras,}
J. Algebra 631 (2023) 148--193.

\bibitem[PV21]{PV1}
M. Porat, V. Vinnikov:
{\it Realizations of non-commutative rational functions around a matrix centre, I: synthesis, minimal realizations and evaluation on stably finite algebras},
J. Lond. Math. Soc. 104 (2021) 1250--1299.

\bibitem[PV23]{PV2}
M. Porat, V. Vinnikov:
{\it Realizations of noncommutative rational functions around a matrix centre, II: the lost-abbey conditions},
Integr. Equ. Oper. Theory 95 (2023) 1--58.

\bibitem[Pro76]{Pro}
C. Procesi:
{\it The invariant theory of $n\times n$ matrices},
Adv. Math. 19 (1976) 306--381.

\bibitem[Row80]{Row}
L. H. Rowen:
{\it Polynomial identities in ring theory},
Pure and Applied Mathematics 84, Academic Press, Inc., New York-London, 1980.

\bibitem[Sha13]{Sha}
I. R. Shafarevich:
{\it Basic algebraic geometry. 1},
3rd edition, Springer, Heidelberg, 2013.

\bibitem[Sou86]{Sou}
A.R. Sourour:
{\it A factorization theorem for matrices},
Linear Multilinear Algebra 19 (1986) 141--147.

\bibitem[Vol17]{Vol}
J. Vol\v{c}i\v{c}:
{\it  On domains of noncommutative rational functions}, Linear Algebra Appl. 516 (2017) 69--81.

\bibitem[Vol21]{Vol1}
J. Vol\v{c}i\v{c}: 
{\it Hilbert's 17th problem in free skew fields}, Forum Math. Sigma 9 (2021) E61.

\end{thebibliography}
\end{document}